%% file: driver.tex
\documentclass[11pt]{article}
\usepackage{epsfig}
\usepackage{amsmath,amssymb,amsthm}
\usepackage{mathtools} 
\usepackage{mathrsfs}   
\usepackage[utf8]{inputenc}

\begin{document}
%%%%%%%%%%%%%%%%%%%%%%%%%%%%%%%%%%%%%%
%
%DEFINITIONS
\newtheorem{theorem}{Theorem}
\newtheorem{thm}{Theorem}
\newtheorem{prop}{Proposition}
\newtheorem{lemma}{Lemma}

\newcommand{\beqn}{\begin{equation}}
\newcommand{\eeqn}{\end{equation}}
\def\R{{\mathbb R}}
\def\D{{\cal D}}
\newcommand{\cleq}{\preccurlyeq}
\newcommand{\nn}{\nonumber}
\def\la{{\langle}}
\def\ra{{\rangle}}
\def\pa{{\partial}}
\def\ep{\epsilon}
\def\F{{\cal F}}
\def\oml{{\ddot{\text{o}}}}

\def\L{{\cal L}}
\def\M{{\mathcal{M}}}

\def\cR{{\cal R}}
\def\Ra{{\cal R}}
\def\Rad{\mathcal{R}} % Radon transform

\def\W{\mathcal{W}}
\def\U{\mathcal{U}}
\def\V{\mathcal{V}}
\def\S{{\mathscr{S}}}
\def\Sdot{{\dot{\S}}}
\def\Cdot{{\dot{C}}}
\def\Ccd{{ \Cdot_c^\infty(\R^n)}}

\title{Recovering initial values from light cone traces of solutions of the wave equation}

\author{
Rakesh \\
Department of Mathematical Sciences\\
University of Delaware\\
Newark, DE 19716, USA\\
~\\
Email: rakesh@udel.edu\\
\and
Tao Yuan\\
TD Securities\\
31 West 52nd Street \#18 \\
New York, NY, USA\\
~\\
Email: tao.yuan@tdsecurities.com
 }

\date{March 29, 2018}
\maketitle

\begin{abstract}
We consider the problem of recovering the initial value, from the trace on the light cone, of the solution of an initial
value problem for the wave equation. 
When the space is odd dimensional, we show that the map from the initial value to the traces of the (even or odd in time)
solutions on the light cone 
is an isometry and we characterize the range of this map and construct its inverse. We do this by relating the problem to 
the recovery of a function from its spherical means over all spheres through the origin, which in turn is related to the 
Radon transform inversion via the inversion map on $\R^n$.
\end{abstract}

%\begin{keywords}
%Goursat problem, spherical means, inverse problem, wave equation
%\end{keywords}

%\begin{AMS}
%35R30, 78A46, 86A22
%\end{AMS}

%%%%%%%%%%%%%%%%%%%%%%%%%%%%%%%%

\input{intro}

\input{preliminaries}

\input{uproof}
\input{vproof}

\input{mean}
\section{Acknowledgments}
A substantial part of this article is based on \cite{yuan}, the PhD thesis of the second listed author, written under the 
supervision of the first listed author. We thank Todd Quinto and Venky Krishnan for discussions about the problem considered in 
this article. We also thank the
referees for a careful reading of the manuscript and pointing out errors which affected some of the 
formulas in the theorems. Rakesh's work was partially supported by NSF grants DMS 0907909 and DMS 1312708.
%
\input{mybib}

%%%%%%%%%%%%%%%%%%%
\end{document}

%% file: intro.tex
% !TEX root = driver.tex
%

\section{Introduction}

\subsection{The Problem}

Suppose $f(x), g(x)$ are smooth functions on $\mathbb{R}^n$, $n>1$, and $w(x,t)$ is the solution of the 
IVP (Initial Value Problem)
\begin{align}
& w_{tt}-\Delta w=0,\quad (x,t)\in \mathbb{R}^n\times \mathbb{R}, \label{eq:wxt} \\
& w(x,0)=f(x),\ w_t(x,0)=g(x),\quad x\in \mathbb{R}^n. \label{eq:wx0}
\end{align}
This is a well posed problem (see \cite{CHII}) and explicit formulas are known for $w$ in terms of $f$ and $g$.  
Define the linear maps
$\W^{\pm} : C^\infty(\R^n) \times C^\infty(\R^n) \mapsto C(\R^n)$ with
\begin{equation}\label{eq:wxx}
 \W^{\pm}(f,g)(x) = w(x, \pm|x|), \qquad x \in \R^n
\end{equation}
which map\footnote{
The traces, $W^{\pm}(f,g)$, may not be smooth at the origin, that is why the codomains for $\W^{\pm}$ are $C(\R^n)$.
}
 the initial data $(f,g)$ to the traces, of the solution $w$, on the light cones $t = \pm |x|$. Our goal is the inversion of the linear 
 map $(f,g) \to (W^+(f,g), W^-(f,g))$.

If $u(x,t)$ and $v(x,t)$ are the even and the odd parts, w.r.t $t$, of $w(x,t)$, that is,
\begin{align*}
u(x,t) =\frac{w(x,t)+w(x,-t)}{2}, \qquad v(x,t) =\frac{w(x,t)-w(x,-t)}{2},
\end{align*}
then $u(x,t), v(x,t)$ are the unique solutions of the IVPs:
\begin{align}
& u_{tt}-\Delta u=0,\quad (x,t)\in \mathbb{R}^n\times \mathbb{R} \label{eq:ude} \\
& u(x,0)=f(x),\ u_t(x,0)=0,\quad x\in \mathbb{R}^n \label{eq:uic}
\end{align}
and 
\begin{align}
& v_{tt}-\Delta v=0,\quad (x,t)\in \mathbb{R}^n\times \mathbb{R} \label{eq:vde} \\
& v(x,0)=0,\ v_t(x,0)=g(x),\quad x\in \mathbb{R}^n \label{eq:vic}
\end{align}
respectively. Define the linear maps $\U : C^\infty(\R^n) \mapsto C(\R^n)$ and $\V : C^\infty(\R^n) \mapsto C(\R^n)$ with
\[
(\U f)(x)=u(x,|x|), \qquad (\V g)(x) = v(x,|x|), \qquad x \in \R^n.
\]

Now
\[
\U f = \left (\frac{ \W^+ + \W^-}{2}  \right )(f,g), \qquad \V g= \left ( \frac{ \W^+ - \W^-}{2} \right )(f,g)
\]
so the inversion of $(f,g) \to (W^+(f,g), W^-(f,g))$ is equivalent to the 
inversion of $\U$ and $\V$. When $n$ is odd, we show that $\U$ and $\V$ are isometries, characterize the ranges of 
extensions of $\U$ and $\V$ and give inversion formulas for $\U$ and $\V$. From this one easily derives an inversion procedure for the map $(f,g) \to (W^+(f,g), W^-(f,g))$; specifically
\[
(W^+(f,g), W^-(f,g) ) \to (\U f, \V g) \to (f,g).
\]
The isometry and the ranges of $\U$ and $\V$ may lead to similar statements about the map $(f,g) \to (W^+(f,g), W^-(f,g))$ 
but they seem cumbersome and we have not explored this issue.

If $f=g$ then one may verify that $u=v_t$ but the connection between the traces of $u$ and $v$ on $t=|x|$, that is 
between $\U$ and $\V$, is cumbersome and we do not use it. We deal with $\U$ and $\V$ independently, though in 
somewhat similar 
fashion, without exploiting the cumbersome connection.

At first glance it may appear that the traces of $w$ on $t=\pm |x|$ do not have enough information to recover $f,g$ since 
most of the initial energy is surely propagated away from the cone $t=\pm |x|$. However, that is incorrect as the propagation
of most of the
initial energy does leave traces on the double cone $t=\pm |x|$ because any ray originating at a point on $t=0$, such as
\[
x = a+ s\theta, ~~ t=s, \qquad -\infty <s < \infty,
\]
for a fixed $a \in \R^n$, $|\theta|=1$, will intersect the double cone $t=\pm |x|$, except for rays in a zero measure 
(for a fixed $a$) set of directions $\theta \perp a$ - see \cite{BK} for another problem where the situation with the rays is
 the same but there is still an isometry.
 
The results are obtained using the following connection. The traces on the cone $t=|x|$, of the solutions of IVPs for the
wave equation, may be written in terms of spherical means of the initial data over spheres through the origin,
 which in turn can be related to the 
Radon transform via the inversion map $x \to x/|x|^2$ which maps spheres through the origin to hyperplanes. We exploit the 
results about the Radon transform to obtain our results.

\subsection{Motivation and history}\label{subsec:history}
Our problem may be regarded as loosely equivalent to the exterior Goursat problem where one studies the well-posedness
of the characteristic boundary value problem for $w(x,t)$ 
\begin{gather*}
\Box w = 0, \qquad  (x,t) \in \R^n \times \R, ~ |x| \geq |t|
\\
w(x,|x|) = \phi(x), ~ w(x,-|x|) = \psi(x), ~ \qquad x \in \R^n
\end{gather*}
where $\phi$, $\psi$ satisfy some matching condition at $x=0$.

The existence, uniqueness and stability of the solution of the above problem in the {\bf interior of the cone} is well studied 
and explicit solutions are available in \cite{CHII}, though the optimal regularity result for $w(x,t)$ at $(x=0,t=0)$ is a tricky matter - see \cite{Cag}. The interior problem and its generalization to other hyperbolic operators have attracted attention because of applications to General Relativity. The exterior problem for more general hyperbolic PDEs is also of interest in General Relativity - see \cite{IK1}, \cite{IK2}, and there are unique continuation results for more general hyperbolic
PDEs for a slightly different problem - see \cite{Friedlander}, \cite{IK1}, \cite{IK2}, \cite{lerner}, \cite{whitman}. We study the
exterior problem only for the wave equation.

Since solutions of the IVP for the wave equation may be expressed in terms of the spherical mean values of the initial data, and since traces on the light cone are related to spherical mean values on spheres through the origin, it is clear that there is a connection between the inversion of $\U$ and $\V$ and the recovery of a function from its spherical mean
values over all spheres through the origin. The results for the spherical mean value inversion problem do not lead directly to results for the $\U, \V$ inversion problem, except for the $n=3$ case. Further, in our opinion, the results for $\U$ and $\V$ seem more compact and aesthetically pleasing.
As far as we know, the inversion of $\U$ and $\V$ (other than the $n=3$ case) 
has received little attention while 
there has been 
considerable work on the spherical mean inversion problem mentioned above - 
see \cite{chen} \cite{CQ}, \cite{moon}, \cite{Quinto}, \cite{Rhee1}, \cite{Rhee2}, \cite{Rhee3}, \cite{Ya}. 
These articles have studied questions such as the recovery of a function from its spherical averages over all spheres through the 
origin or determining all functions which have zero spherical averages, over all spheres through the origin which lie in a fixed ball 
through the origin. 

Our main tool to tackle the $\U, \V$ problem - using the inversion map $x \to x/|x|^2$ to relate the problem 
to the Radon transform inversion - was already used in some of the articles mentioned above 
(also see \cite{rubinwang}). However, we believe, even our Theorem \ref{thm:mean} about spherical mean inversion is new. 
Our work is mainly directed 
towards the $\U, \V$ inversion problem for which 
we have new results - our Theorems \ref{thm:isometry}, \ref{thm:first}, \ref{thm:second}. 

The problem of recovering a function from its spherical averages over a family of spheres has a long history. The problem of 
recovering a function from its spherical averages over all spheres with a fixed radii or over spheres with centers on a plane or 
centers on a sphere etc.~have been studied and many of these problems have applications to medical imaging. The field is too 
broad for a general survey article and we suggest the introduction of \cite{FPR} as perhaps a reasonable starting point 
for the literature on such problems.

\subsection{Main Results}

For any positive integer $n$, $S^{n-1}$ will denote the unit sphere in $\R^n$, $\omega_{n-1}$ its surface area, 
$\pa_r = |x|^{-1} x \cdot \nabla$ will be the radial derivative
and $\delta^m(t)$ will denote the m-th derivative of the Dirac delta distribution.  For any surface $\Gamma$ in $\R^n$, $dS$
denotes the surface measure on $\Gamma$, sometimes written as $dS_x$ if the variable $x$ represents an arbitrary point on 
$\Gamma$.
Further, $\S(\R^n)$ will denote the Schwartz space, 
$C_c^\infty(\R^n)$ will denote the set of compactly supported smooth functions on $\R^n$ and
\begin{align*}
\Cdot^\infty(\R^n) & =  \{h(x) \in C^\infty(\R^n): 0 \notin \text{supp}~ h \},
\\
\Ccd & = \{h(x) \in C_c^\infty(\R^n): 0 \notin \text{supp}~ h \}.
\end{align*}
 For any $h \in \S(\R^n)$, its Radon transform is
 \[
 (\Rad h)(\theta,s) := \int_{x \cdot \theta =s} f(x) \, dS_x = \int_{\R^n} h(x) \, \delta( x \cdot \theta -s) \, dx
 , \qquad \theta \in \R^n, ~ |\theta|=1, ~ s \in \R
 \]
and its spherical mean value on a sphere centered at $c \in \R^n$ of radius $t \in \R$ is defined as
 \begin{align}
 (\M h)(c, t) & := \frac{1}{\omega_{n-1}} \int_{S^{n-1}} h(c+t \theta) \, d \theta,
 \nonumber
 \\
 & =  \frac{1}{\omega_{n-1} t^{n-1} } \int_{|x-c|=t} h(x) \, dS_x
 = \frac{2}{\omega_{n-1} t^{n-2} } \int_{\R^n} h(x) \, \delta( t^2 - |x-c|^2 ) \, dx, \qquad t>0.
 \label{eq:Mhd}
 \end{align}
 For any non-negative measurable function $\rho$ on $\R^n$, we define $ L^2(\R^n, \rho)$ to be the Hilbert space of all 
 measurable $f : \R^n \to \R$ for which $ \int_{\R^n} \rho(x) \, |f(x)|^2 \, dx $ is finite.

For $n \geq 3$, $n$ odd, we show that $\U$ and $\V$ are isometries, give inversion formulas, and characterize the 
ranges of the
extensions of these maps. {\bf The problem is unresolved for the even $n$ case} though we have unpublished partial results
for the even $n$ case.

\begin{theorem}[Isometry and Range] \label{thm:isometry}
If $n \geq 3$, $n$ odd, then for all $f, g \in \Ccd$ we have  $\U f, \V g  \in \Cdot^\infty(\R^n)$ and
\begin{align*}
\int_{\mathbb{R}^n} \frac{|f(x)|^2}{|x|^2} \; dx
& =2\int_{\R^n} \frac{|(\U f)(x)|^2}{|x|^2} \; dx,
\\
\int_{\R^n} |x|^2 \, |g(x)|^2 \; dx
& =8\int_{\R^n} |x|^{3-n}  \, |  \partial_r(|x|^{\frac{n-1}{2}}(\V g)(x))  |^2\; dx.
\end{align*}
Further, the map $f \to \U f$ has a continuous linear extension as a {\bf bijection} from $L^2(\R^n, |x|^{-2})$ to itself, and the map
$g \to |x|^{(1-n)/2} \pa_r ( |x|^{(n-1)/2} (\V g)(x) )$ has a continuous linear extension as a {\bf bijection} from
$L^2(\R^n, |x|^2)$ to itself.
\end{theorem}

We have inversion formulas for $\U$ and $\V$ closely connected to the inversion formula for the Radon transform.
\begin{theorem}[First inversion formula]\label{thm:first}
If $n \geq 3$ and $n$ is odd with $n=2m+1$ then for all $f,g \in \Ccd$ we have
\begin{align*}
f(x)& =\frac{1}{(4\pi)^m |x|^{2m} }\int_{S^{n-1}}\pa_s^m\left( \frac{1}{s^m}(\U f)(\frac{\theta}{2s})
\right)\bigg|_{s=x\cdot\theta/|x|^2}\; d\theta, \qquad
x \in \R^n, ~ x \neq 0,\\
g(x)& =\frac{-1}{(4\pi)^m |x|^{n+1}}\int_{S^{n-1}} \pa_s ^{m+1}
\left( \frac{1}{s^{m-1}|s|}(\V g)(\frac{\theta}{2s})\right)
\bigg|_{s=x\cdot\theta/|x|^2} \; d\theta,
\qquad
x \in \R^n, ~ x \neq 0.
\end{align*}
\end{theorem}

We derive a second set of inversion formulas, coming from the isometries and the
adjoints of $\U$ and $\V$ with respect to the associated inner products.
\begin{theorem}[Second inversion formula] \label{thm:second}
If $n \geq 3$ and $n$ is odd with $n=2m+1$ then we have
\begin{align*}
f(x) &= 2 \,\U^*(\U f)(x),
\qquad \forall x \in \R^n, x \neq 0, ~ f \in \Ccd
\end{align*}
where $\U^*$ is (below  $\phi_*(y)=|y|^{-1} \phi(y)$)
\[
(\U^* \phi)(x) = \frac{1 }{2 (-2\pi)^m |x|^{m-1}}  (\partial_s^m \Rad \phi_*)(x/|x|, |x|/2),
\qquad  x \neq 0, ~ \phi\in \Ccd
\]
or its continuous linear extension as an isometry from $L^2(\R^n, |x|^{-2})$ to itself. Also,
\begin{align*}
g(x) & = 8 \V^*( |x|^{-m} \pa_r (|x|^m (\V g)(x))(x),
\qquad  x \neq 0, ~ g \in \Ccd
\end{align*}
where $\V^*$ is (below $\phi^*(y) = |y| \phi(y)$)
\[
(\V^* \phi)(x) = \frac{1}{4  (-2 \pi)^m |x|^{m+1} } ( \pa_s^m \Rad \phi^*)(x/|x|, |x|/2),
 \, \qquad x \neq 0, ~ \phi \in \Ccd
\]
or its continuous linear extension as an isometry from $L^2(\R^n, |x|^2)$ to itself.
\end{theorem}

The inversion of $\U$ and $\V$ has a connection with the problem of recovering a function from its spherical averages over all 
spheres through the origin, that is, given $(\M f)(x,|x|)$ for all $x \in \R^n$ - recover $f$. The spherical average problem 
can be tackled by the same methods
as those applied to the $\U, \V$ problem, except the spherical average problem is a little easier but the results are more
cumbersome than those for the $\U, \V$ problem. 
When $n=3$, the spherical average problem and the $\V$ problem are equivalent; for 
other $n$ the connection is complicated and there is no simple path to obtain results for one problem from the other.
This spherical average inversion 
problem has received a fair amount of attention but the results are 
incomplete (see subsection \ref{subsec:history}). 
For the spherical average problem, we give the results only for the odd $n$ 
case but our technique works also for the even $n$ case - the results for the even $n$ case are not as appealing as 
the 
odd $n$ case.

%
%%%%%%%%%%%%%%%%%%%%%%%
\begin{theorem}[Isometry and inversion for spherical means] \label{thm:mean}
Suppose $n$ is odd and $h(x)\in \Ccd$. We have the isometry
\[
\int_{\R^n} |x|^{2n-4} \, |h(x)|^2 \, dx =
\frac{2 \pi}{\Gamma(n/2)^2}  \int_{\R^n}  | (\rho^2 \pa_\rho)^{(n-1)/2} ( \rho^{n-1} (\M h)(y, |y|) ) |^2  \, |y|^{-n-1} \, dy
\]
and the inversion formula
\[
h(x) = \frac{ (-1)^{(n-1)/2} \omega_{n-1}}{ 2\, \pi^{n-1}} |x|^{3-2n} \,
\int_{2 y \cdot x = |x|^2}  |y|^{-n} \, ( \rho^2 \partial_ \rho)^{n-1} ( \rho^{n-1}  (\M h)( y, |y|) ) \, dS_y
\]
where $\rho = |y|$.
\end{theorem}
Note the integration on the RHS of the inversion formula - the integral is over ``all spheres which pass through $x$ and the origin''.

%% file: preliminaries.tex
% !TEX root = driver.tex

\section{Preliminaries}
We introduce notation and state and prove preliminary results needed in the proofs of the theorems.
We define $\S(S^{n-1} \times \R)$ to consist of
functions in $\S(\R^n \times \R)$ restricted to $S^{n-1} \times \R$ and
\begin{align*}
\S_e(S^{n-1} \times \R) &=  \{  h \in \S(S^{n-1} \times \R) \, : \, h(-\theta, -s) =  h(\theta,s), ~ \forall
(\theta,s) \in S^{n-1} \times \R \},
\\
L^2_e(S^{n-1} \times \R) & = \{  h \in L^2(S^{n-1} \times \R) \, : \, h(-\theta, -s) =  h(\theta,s), ~ \forall
(\theta,s) \in S^{n-1} \times \R \},
\\
L^2_\sigma(S^{n-1} \times \R) & = \{  h \in L^2(S^{n-1} \times \R) \, : \, h(-\theta, -s) = (-1)^{(n-1)/2} h(\theta,s), ~ \forall
(\theta,s) \in S^{n-1} \times \R \}.
\end{align*}
We also define the operator $D = \frac{1}{2t} \pa_t$, which acts like differentiation w.r.t $t^2$ because, for any differentiable 
function $h(t)$, we have
\[
D( h(t^2) ) = h'(t^2), \qquad \forall t \in \R, ~ t \neq 0.
\]

Since it is easier to manipulate integrals on $\R^n$ than surface integrals, sometimes we convert surface integrals to integrals on
$\R^n$ using the distributional relation
\beqn
\int_{\phi=0} \frac{f(x)}{| (\nabla \phi)(x)|}  \, dS_x = \int_{\R^n} f(x) \, \delta( \phi(x) ) \, dx
\label{eq:surfdelta}
\eeqn
for any $f \in C_c^\infty(\R^n)$ and any $\phi \in C^\infty(\R^n)$ with $(\nabla \phi)(x) \neq 0$ if $\phi(x)=0$.

%%%%%%%%%%%%%%%%%%%%%
The inversion map on $\R^n$ plays an important role in the proofs because inversion maps spheres through the origin to hyperplanes.
\begin{prop}[Properties of inversion]\label{prop:reflection}
The inversion map $x \to X = x/|x|^2$ on $\R^n  \setminus \{0\}$ has the following properties:
\vspace{-0.3in}
\begin{enumerate}
\item[(a)] For any $c \in \R^n$, $c \neq 0$, the sphere $|x-c|=|c|$ is mapped to the hyperplane $ 2 X \cdot c = 1$;
\item[(b)] $dX = |x|^{-2n} \, dx = |X|^{2n} dx$;
\item[(c)] The map $ h(x) \to H(X) = |X|^k h(X/|X|^2)$, extended by zero, is a linear bijection from 
$\Ccd$ to itself, for every integer $k$.
\end{enumerate}
\end{prop}

\begin{proof}
We have $x = X/|X|^2$ and hence 
\[
|x-c|^2 - |c|^2 = |x|^2 - 2 x \cdot c = \frac{1}{|X|^2} - 2 \frac{X \cdot c}{|X|^2} = \frac{1 - 2 X \cdot c}{|X|^2}
\]
which proves (a).
Using spherical coordinates $(\rho_x=|x|, \theta = x/|x|)$ for $x$ and $(\rho_X=|X|, \theta= X/|X|)$ for $X$ we have 
$\rho_X = \rho_x^{-1}$ so $ d \rho_X / d \rho_x = - \rho_x^{-2} $ and hence
\[
dX = \rho_X^{n-1} d \rho_X \, d \theta = \rho_x^{1-n} \, | -\rho_x^{-2}| \, d \rho_x \, d \theta = \rho_x^{-n-1} d \rho_x \, d \theta
= \rho_x^{-2n} dx = |x|^{-2n} dx
\]
proving (b).
(c) follows easily from the definition of $\Ccd$.
\end{proof}

For use later, we recall the standard properties of the Radon transform in odd dimensions, found in 
Chapter 1 of \cite{GGV} and Chapter 1 of \cite{helgason}.
\begin{theorem}[Properties of the Radon transform]\label{thm:radon}
For $n \geq 3$, $n$ odd, the Radon transform $\Rad$ is an injective linear map from
 $\S(\R^n)$ to $\S_e( S^{n-1} \times \R)$ and has the following properties:
\vspace{-0.3in}
\begin{enumerate}
\item[(a)] (Equation (1') on page 12 in \cite{GGV}) $\Rad$ is an isometry with
\[
\int_{\R^n} |h(x)|^2 \, dx = \frac{1}{2 (2 \pi)^{n-1}} \int_{S^{n-1}} \int_\R \left | \pa_s^{(n-1)/2} ( \Rad h)(\theta,s)
\right |^2 \, ds \, d \theta ,
\qquad \forall h \in \S(\R^n);
\]
\item[(b)] (Theorem 3.6 in Chapter 1 of \cite{helgason}) For any $h \in \S(\R^n)$ we have
\[
h(x) = \frac{ (-1)^{(n-1)/2} }{ 2 (2 \pi)^{n-1} } \int_{S^{n-1}} \pa_s^{n-1} ( \Rad h)(\theta, s)|_{s = x \cdot \theta} \, d \theta,
\qquad \forall x \in \R^n;
\]
\item[(c)] (Remark on page 14/15 in \cite{GGV} and Theorem 4.1 on page 21 in \cite{helgason}) The isometry 
$h \to  \pa_s^{(n-1)/2} \Rad h$, from 
$\S(\R^n)$ to $\S(S^{n-1} \times \R)$, has a continuous linear extension as a {\bf bijection} from $L^2(\R^n)$ to $L^2_\sigma(S^{n-1} \times \R)$. 
\end{enumerate}
\end{theorem}

%% file: uproof.tex
% !TEX root = driver.tex

 \section{The inversion of $\mathcal{U}$}
 
 In this section we provide the proofs of the parts of Theorems \ref{thm:isometry} - \ref{thm:second} pertaining to $\U$. Below
 $n=2m+1$ with $m \geq 1$.

%%%%%%%%%%%%%%%%%%%%
\subsection{Proof of Theorem \ref{thm:isometry} for $\U$}

If $f \in \Ccd$ then, from the standard theory, the IVP (\ref{eq:ude}), (\ref{eq:uic})
has a unique solution $u(x,t) \in C^\infty(\R^n \times \R)$. In addition, since 
$f=0$ in a neighborhood of $0$, we have $u(x,t)=0$ in neighborhood of $(0,0)$ implying $\U f \in \dot{C}^\infty(\R^n)$.

If $f \in \Ccd$ then (see page 682 of \cite{CHII})
\[
u(x,t)= \frac{\sqrt{\pi}}{\Gamma(\frac{n}{2})} t D^{(n-1)/{2}} \left(t^{n-2}(\mathcal{M}f)(x,|t|)\right), \quad \forall 
(x,t)\in \R^n\times \R, ~ t \neq 0.
\]
When $t>0$, noting that $\omega_{n-1} = \frac{2 \pi^{n/2}}{\Gamma(n/2)}$ and using (\ref{eq:Mhd}), we have
\begin{align*}
u(x,t)
&= \frac{t}{\pi^m} D^m \left( \int_{\R^n} f(y)  \, \delta(t^2 - |y-x|^2)\; dy  \right) \\
&= \frac{t}{\pi^{m}}\int_{\R ^n} f(y) \, \delta^{m}(t^2 - |y-x|^2)\; dy
\end{align*}
so 
\begin{align}
(\U f)(x)=\frac{|x|}{\pi^m}\int_{\R^n} f(y) \, \delta^{m}(2x\cdot y - |y|^2)\; dy, 
\qquad \forall x\in \R^n, \; x \neq 0.
\label{eq:ufxn}
\end{align}
For $x\in \mathbb{R}^n$, $x\not = 0$,  let $ \theta=x/|x|, s=\frac{1}{2|x|}$ and define
\begin{equation*}
F(X)=f(X/|X|^2) \,|X|^{1-n};
\end{equation*}
then by Proposition \ref{prop:reflection}, $F(X)\in \Ccd)$. Using the homogeneity of $\delta^k(t)$ we have
\begin{align*}
2^{m+1}\pi^m |x|^m ( \U f)(x)&=\int_{ \R^n} f(y) \, |2x|^{m+1} \,\delta^{m}(2 x \cdot y - |y|^2)\; dy
\\
&= \int_{\R^n}  \frac{f(y)}{|y|^{2m+2}} \, \delta^{m} \left ( \frac{x }{|x|}\cdot \frac{y}{|y|^2} - \frac{1}{2|x|} \right ) \; dy
\\
&= \int_{\R^n} \frac{f(y)}{|y|^{2m+2}} \, \delta^{m} \left ( \theta\cdot \frac{y}{|y|^2} - s \right )\; dy
\tag*{let $Y=y/|y|^2$}
\\
&= \int_{\R^n} F(Y) \, \delta^{m}(\theta\cdot Y- s)\; dY
\\
&= (-1)^m \partial_s^{m}(\Rad F)(\theta,s),
\end{align*}
that is
\begin{equation}
 \partial_s^{m} (\mathcal{R}F)(\theta,s)
 =\frac{(-1)^m 2 \pi ^m}{s^m}(\mathcal{U}f)(\frac{\theta}{2s}), \quad \forall \theta\in S^{n-1}, \; s>0.
 \label{eq:RFUf}
\end{equation}

Since $F \in \Ccd$ and $(\Rad F)(-\theta, -s) = (\Rad F)(\theta, s)$, from (\ref{eq:RFUf}) and Theorem \ref{thm:radon}
 we have
\begin{align*}
\int_{\R^n}|F(X)|^2\; dX
& =\frac{1}{2(2\pi)^{n-1}}\int_{S^{n-1}}\int_{-\infty}^\infty |\partial_s^{m}(\Rad F)(\theta, s)|^2 \; ds  \; d\theta
\\
&=\frac{1}{(2 \pi)^{n-1}}\int_{S^{n-1}}\int_{0}^\infty |\partial_s^{m}(\mathcal{R}F)(\theta, s)|^2 \; ds \; d\theta 
\\
&=\frac{1}{2^{n-3}}\int_{S^{n-1}}\int_{0}^\infty s^{1-n} |(\U f)(\theta/(2s))|^2 \; ds \; d\theta 
\tag*{let $r=\frac{1}{2s}$}
\\
&=2\int_{S^{n-1}}\int_{0}^\infty r^{n-3}|u(x,|x|)|^2 \; dr \; d\theta
\\
&=2\int_{\mathbb{R}^n}\left(\frac{u(x,|x|)}{|x|}\right)^2\; dx.
\end{align*}
Now, from Proposition \ref{prop:reflection}
\begin{equation*}
\int_{\mathbb{R}^n}|F(X)|^2\; dX=\int_{\mathbb{R}^n}(f(x)|x|^{n-1})^2\; |x|^{-2n} \, dx=\int_{\mathbb{R}^n}\left(\frac{f(x)}{|x|}\right)^2\; dx,
\end{equation*}
so the proof of the isometry is complete.

From Theorem \ref{thm:radon} and that $\Ccd$ is dense in $L^2(\R^n)$ the map 
\[
 F \to s^m \pa_s^m (\Rad F)(\theta,s)
\]
from
$\Ccd$ to $\S_e(S^{n-1} \times \R)$ has an extension as a continuous linear bijection from $L^2(\R^n)$ to
$L^2_e(S^{n-1} \times \R, s^{-2m})$. So from proposition \ref{prop:reflection}, the map 
\[
f(x) \to F(X) \to s^m \pa_s^m (\Rad F)(\theta,s)
\]
 from $\Ccd$ to $\S_e(S^{n-1} \times \R)$ has an
extension as a continuous linear bijection from $L^2(\R^n, |x|^{-2})$ to $L_e^2(S^{n-1} \times \R, s^{-2m})$.

Given $h : \R^n \to \R$, define $\tilde{h} : S^{n-1} \times (\R \setminus \{0\}) \to \R$ with
$\tilde{h}(\theta,s) = h(\theta/(2s))$. Since
\begin{align*}
\int_{S^{n-1}} \int_\R | \tilde{h}(\theta,s)|^2 \, s^{-2m} \, ds \, d \theta
& = \int_{S^{n-1}} \int_0^\infty | h(\theta/(2s))|^2 \, s^{-2m} \, ds \, d \theta
+ \int_{S^{n-1}} \int_{-\infty}^0 | h(\theta/(2s))|^2 \, s^{-2m} \, ds \, d \theta
\\
&= \frac{1}{2} \int_{S^{n-1}} \int_0^\infty | h(r \theta)|^2 \, \frac{r^{2m}}{r^2}  \, dr \, d \theta
+ \frac{1}{2} \int_{S^{n-1}} \int_0^\infty | h(r \theta)|^2 \, \frac{r^{2m}}{r^2}  \, dr \, d \theta
\\
&= \int_{\R^n} \ \frac{|h(x)|^2}{|x|^2} \, dx,
\end{align*}
the map $ h \to \tilde{h}$ is a continuous linear bijection from $L^2(\R^n, |x|^{-2})$ to 
$L^2_e(S^{n-1} \times \R, s^{-2m})$.

From (\ref{eq:RFUf}) we have
\[
s^m  \partial_s^{m} (\mathcal{R}F)(\theta,s)
 =c \, (\mathcal{U}f)(r \theta )|_{r=1/(2s)}, \quad \forall \theta\in S^{n-1}, \; s>0,
\]
so using the results in the previous two paragraphs,  we conclude that the map $f \to \U f$ has an extension 
as a bijection from $L^2(\R^n, |x|^{-2})$ to itself.

 %%%%%%%%%%%%%%%%%%%%%%%%%%%%%%%%%%%%
 %%%%%%%%%%%%%%%%%%%%%%%%%%%%%%%%%%%
 \subsection{Proof of Theorem \ref{thm:first} for $\U$}
 
 If $f \in \Ccd$ and we define $F(X) = f(X/|X|^2) \, |X|^{1-n}$ then $F \in \Ccd$ and from (\ref{eq:RFUf}) we have
 \[
 \partial_s^{m} (\Rad F)(\theta,s)
 =\frac{(-1)^m 2 \pi ^m}{s^m}(\U f)(\frac{\theta}{2s}), \quad \forall \theta\in S^{n-1}, \; s>0.
 \]
When $s<0$, substituting $(-\theta,-s)$ into the above identity and noting $(\Rad F)(-\theta, -s) = (\Rad F)(\theta, s)$, we obtain
\[
 \partial_s^{m} (\Rad F)(\theta,s)=\frac{(-1)^m 2\pi^m}{s^m}(\U f)(\frac{\theta}{2s}), \quad \forall \theta\in S^{n-1}, \ s< 0.
\]
Hence
\[
 \partial_s^{m} (\Rad F)(\theta,s)=\frac{(-1)^m 2\pi^m}{s^m}(\U f)(\frac{\theta}{2s}), 
 \quad \forall \theta\in S^{n-1}, \ s\in \R, ~ s \neq 0,
\]
so
\begin{equation*}
 \partial_s^{2m} (\Rad F)(\theta,s)=(-1)^m 2\pi^m\partial_s^m\left( \frac{1}{s^m}(\U f)(\frac{\theta}{2s})\right), \quad \forall \theta\in S^{n-1}, \ s\in \R, \, s \neq 0
\end{equation*}
Hence, using the Radon transform inversion formula (see Theorem \ref{thm:radon}), for $X \neq 0$, we have
\begin{align*}
F(X) 
& =\frac{(-1)^m}{2 (2 \pi)^{2m}} \int_{S^{n-1}} \partial_s^{2m} (\Rad F)(\theta,s)|_{s=X\cdot\theta}\; d\theta
\\
& = \frac{1}{(4\pi)^m}\int_{S^{n-1}}\partial_s^m\left( \frac{1}{s^m}(\mathcal{U}f) 
(\frac{\theta}{2s})\right)\bigg|_{s=X\cdot\theta}\; d\theta,
\end{align*}
hence
\begin{align*}
f(x) & =\frac{1}{(4\pi)^m |x|^{n-1} }\int_{S^{n-1}}\partial_s^m\left( \frac{1}{s^m}(\U f)(\frac{\theta}{2s})
\right)\bigg|_{s=x\cdot\theta/|x|^2}\; d\theta,
\qquad \forall x \neq 0,
\end{align*}
proving the theorem.

%%%%%%%%%%%%%%%%%%%%%%%%%%%%%%%%%%%%%%%%%%%%%
%%%%%%%%%%%%%%%%%%%%%%%%%%%%%%%%%%%%%%%%%

\subsection{Proof of Theorem \ref{thm:second} for $\U$}

We first find the adjoint of $\U$ in the weighted $L^2$ norm.
%%%%%%%
\begin{prop}[\bf The adjoint of $\U$]\label{prop:adjointun}
If $n$ is odd then for any $f, \phi \in \Ccd$ we have 
\[
\int_{\R^n} \frac{ (\U f)(x) \, \phi(x)}{|x|^2} \, dx = \int_{\R^n} \frac{ f(x) \, (\U^* \phi)(x)}{|x|^2} \, dx
\]
where
\begin{equation*}
(\U^* \phi)(x)= \frac{1 }{2 (-2\pi)^m |x|^{m-1}}  (\partial_s^m \Rad \phi_*)(x/|x|, |x|/2),
\quad \forall \phi\in \Ccd,  ~~ x \in \R^n,~ x \neq 0,
\end{equation*}
and $\phi_*(y)=\phi(y) \,|y|^{-1}$. Further
\[
\int_{\R^n} \frac{ | \phi(x)|^2}{|x|^2} \, dx = 2 \int_{\R^n} \frac{ |(\U^* \phi)(x)|^2}{|x|^2} \, dx, \qquad 
\forall \phi \in \Ccd
\]
and the map $\U^* : \Ccd \to L^2(\R^n, |x|^{-2})$ has a continuous linear extension as a {\bf bijection} from
$L^2(\R^n, |x|^{-2})$ to itself.
\end{prop}

%%%%%%%
\begin{proof}
Below, for any $x \in \R^n$, $x \neq 0$, we define $r=|x|$ and $\theta = x/|x|$. From (\ref{eq:ufxn}) we have
\begin{equation*}
(\U f)(x)= \frac{(-1)^m |x|}{\pi^m}\int_{\R^n} f(y) \, \delta^m(|y|^2-2x\cdot y)\; dy,
\qquad \forall f \in \Ccd, ~~ x \neq 0,
\end{equation*}
so for any $\phi \in \Ccd$ we have
\begin{align*}
\int_{\R^n} \frac{ (\U f)(x) \, \phi(x)}{|x|^2} \, dx
& = \frac{(-1)^m}{\pi^m}  \int_{\R^n} |x|^{-1} \phi(x) \int_{\R^n} f(y) \; \delta^m (|y|^2-2x \cdot y) \; dy \; dx \\
& =\frac{(-1)^m}{\pi^m} \int_{\R^n} f(y) \int_{\R^n} |x|^{-1} \phi(x) \; \delta^m (|y|^2-2x \cdot y) \; dx \; dy,
\\
& = \int_{\R^n} \frac{ f(x) \, (\U^* \phi)(x)}{|x|^2} \, dx ,
\end{align*}
where, for $x \neq 0$, we define $\phi_*(x) = |x|^{-1} \phi(x)$ and define
\begin{align*}
(\U^*\phi)(x) 
& =\frac{(-1)^m |x|^2}{\pi^m}  \int_{\R^n} \phi_*(y) \, \delta^m (|x|^2-2x \cdot y) \; dy
=\frac{ r^2}{(-\pi)^m} \int_{\R^n} \phi_*(y) \, \delta^m (r^2-2r\theta \cdot y) \; dy
\\
&=\frac{r }{2 (-2\pi r)^m } \int_{\R^n} \phi_*(y) \, \delta^{m} (r/2 -\theta \cdot y) \; dy 
=\frac{r }{2 (-2\pi r)^m }  (\partial_s^m \Rad \phi_*)(\theta,r/2).
\end{align*}

Further, using Theorem \ref{thm:radon}, we have
\begin{align*}
4 (2 \pi)^{2m} \, \int_{\R^n} \frac{ |( \U^* \phi)(x)|^2}{|x|^2} \, dx
& = \int_{|\theta|=1} \int_0^\infty r^{-2m} \,  |\pa_s^m (\Rad \phi_*)(\theta, r/2)|^2 \, r^{n-1} \, dr \, d \theta
\\
& =  2 \int_{|\theta|=1} \int_0^\infty |\pa_s^m (\Rad \phi_*)(\theta, s)|^2 \, ds \, d \theta
= \int_{|\theta|=1} \int_{\R} |\pa_s^m (\Rad \phi_*)(\theta, s)|^2 \, ds \, d \theta
\\
& = 2 (2 \pi)^{2m} \int_{\R^n} |\phi_*(x)|^2 \, dx
= 2 (2 \pi)^{2m} \int_{\R^n} \frac{ |\phi(x)|^2}{|x|^2} \, dx.
\end{align*}

The map $\phi \to \phi_*$ is a linear bijection on $\Ccd$ and also from $L^2(\R^n, |x|^{-2})$ to $L^2(\R^n)$.
Further, from Theorem \ref{thm:radon}, the map $\phi_* \to s^{1-m} \pa_s^m (\Rad \phi_*)(\theta, s)$ has a continuous
linear extension as a bijection from $L^2(\R^n)$ to $L^2_o(\S^{n-1} \times \R, s^{n-3})$ (here "o" stands for odd in $(\theta,s)$) 
and the map
\[
h(\theta,s) \to \tilde{h}(x) = h(x/|x|, |x|)
\]
is a continuous linear bijection from $L^2_o(\S^{n-1} \times \R, s^{2m-2})$ to $L^2(\R^n, |x|^{-2})$ so the statement about the
extension follows.
\end{proof}

We now continue with the proof of Theorem \ref{thm:second} for $\U$. By Theorem \ref{thm:isometry} and the continuous 
extension of $\U$ we have 
\[
2 \int_{\R^n} \frac{((\U f)(x) \, (\U g) (x)}{|x|^2} \, dx =  \int_{\R^n} \frac{ f(x) \, g(x)}{|x|^2} \, dx,
\qquad \forall f, g \in L^2(\R^n, |x|^{-2})
\]
and from Proposition \ref{prop:adjointun}, the density of $\Ccd$ in $L^2(\R^n, |x|^{-2})$, and the continuous linear extension
of $\U^*$ we have
\[
\int_{\R^n} \frac{(\U f) (x) \, \phi(x)}{|x|^2} \, dx = \int_{\R^n} \frac{ f(x) \, (\U^* \phi)(x)}{|x|^2} \, dx,
\qquad \forall f, \phi \in L^2(\R^n, |x|^{-2}).
\]
Hence applying the first relation to $\phi = \U g$ we conclude that
\[
2 \U^* \U g = g, \qquad \forall g \in L^2(\R^n, |x|^{-2}).
\]

%% file: vproof.tex
% !TEX root = driver.tex
%

\section{The inversion of $\V$}

We give the proofs of Theorems \ref{thm:isometry} - \ref{thm:second} for $\V$. Below $n=2m+1$ for some integer $m \geq 1$

\subsection{Proof of Theorem \ref{thm:isometry} for $\V$}

If $g \in \Ccd$ then the IVP (\ref{eq:vde})-
(\ref{eq:vic}) has a unique smooth solution and, as argued for the $\U$ case, $\V g \in \dot{C}^\infty(\R^n)$. Further
(see page 682 of \cite{CHII})
\begin{equation*}
v(x,t) = \frac{\sqrt{\pi}}{2\Gamma(\frac{n}{2})} D^{(n-3)/2}
\left( t^{n-2}(\M g)(x,|t|) \right), \quad \forall (x,t)\in \R^n\times \R, ~ t \neq 0
\end{equation*}
so, for $t>0$, we have
\begin{align*}
v(x,t)
&= \frac{1}{2 \pi^m}  D^{m-1}
 \left(\int_{\R^n} g(y) \, \delta(t^2 - |y-x|^2)\; dy\right)\\
&= \frac{1}{2 \pi^m} \int_{\R^n} g(y) \, \delta^{m-1}(t^2 -|y-x|^2)\; dy,
\end{align*}
and hence
\begin{align*}
(\V g)(x)= v(x, |x|) &= \frac{1}{2 \pi^m }
\int_{\R^n} g(y)\, \delta^{m-1}(|x|^2 - |y-x|^2)\; dy\\
&= \frac{1}{2 \pi^m} \int_{\R^n} g(y) \, \delta^{m-1}(2 x \cdot y - |y|^2)\; dy,
\qquad \forall x \in \R^n, ~ x \neq 0.
\end{align*}
Define $G(X)=g(X/|X|^2) \,|X|^{-n-1}$ and note that $G(X)\in \Ccd$  by proposition \ref{prop:reflection}. 
For $x \in \R^n$, $x \neq 0$, let $\theta=x/|x|$, $s = \frac{1}{2 |x|}$ - we have
\begin{align*}
2 \pi^m |2x|^{m} (\V g)(x)
&  =\int_{ \R^n} g(y)\, |2x|^m \, \delta^{m-1}(2x\cdot y - |y|^2)\; dy
\\
& = \int_{\R^n} \frac{g(y)}{|y|^{2m}} \, \delta^{m-1} \left (  \frac{x}{|x|}\cdot \frac{y}{|y|^2} - \frac{1}{2|x|} \right )\; dy
\\
&= \int_{\R^n} \frac{g(y)}{|y|^{2m}} \, \delta^{m-1} \left (\theta\cdot \frac{y}{|y|^2} -s \right )\; dy,
\tag*{use $Y=y/|y|^2$}
\\
&= \int_{\R^n} G(Y) \, \delta^{m-1}(\theta\cdot Y -s)\; dY, 
\\
&= (-1)^{m-1} \partial_s^{m-1}(\Rad G)(\theta,s).
\end{align*}
Hence
\beqn
\pa_s^{m-1}(\Rad G)(\theta,s) =\frac{(-1)^{m-1} 2\pi ^m}{s^m} (\V g)(\frac{\theta}{2s}),
\qquad \forall s>0, \, |\theta|=1,
\label{eq:RGVgs}
\eeqn
so, if we use $r=\frac{1}{2s}$ then
\beqn
\pa_s^m (\Rad G)(\theta,s) 
= 4 (-2 \pi)^m \, r^2 \, \pa_r ( r^m (\V g)(r \theta) ), \qquad \forall r>0, ~ |\theta|=1.
\label{eq:RGVg}
\eeqn
%%%
%
Since $G \in \Ccd$ and $(\Rad G)(-\theta, -s) = (\Rad G)(\theta, s)$, from Theorem \ref{thm:radon}
we have
\begin{align*}
\int_{\R^n}|G(X)|^2\; dX 
& = \frac{1}{2(2\pi)^{2m}}\int_{S^{n-1}}\int_{-\infty}^\infty |\pa_s^{m}(\Rad G)(\theta, s)|^2 \; ds \; d\theta
\\
& =  \frac{1}{(2\pi)^{2m}} \int_{S^{n-1}}\int_0^\infty |\pa_s^{m}(\Rad G)(\theta, s)|^2 \; ds \; d\theta
\\
& = 16  \int_{S^{n-1}}\int_0^\infty r^4  \,|\pa_r \left ( r^m (\V g)(r\theta) \right )|^2 \; ds \; d\theta
\qquad \text{here ~} r=\frac{1}{2s}
\\
& = 8 \int_{S^{n-1}}\int_0^\infty r^2 \, |\pa_r \left ( r^m (\V g)(r\theta) \right )|^2 \; dr \; d\theta
\\
& = 8 \int_{\R^n} r^2  \, | r^{-m} \pa_r \left ( r^m (\V g)(x) \right )|^2 \; dx,
\qquad \qquad \text{here~} r=|x|.
\end{align*}
We also have
\[
\int_{\R^n} |G(X)|^2 \, dX = \int_{\R^n} (|x|^{n+1} g(x) )^2 \, |x|^{-2n} \, dx 
= \int_{\R^n} |x|^2 \, |g(x)|^2 \, dx;
\]
hence we obtain the isometry
\[
\int_{\R^n} |x|^2 \, |g(x)|^2 \, dx
= 8 \int_{\R^n}  r^2 \,  | r^{-m} \pa_r \left ( r^m (\V g)(x) \right )|^2 \; dx,
\qquad \qquad \text{here~} r=|x|.
\]

From Theorem \ref{thm:radon} and that $\Ccd$ is dense in $L^2(\R^n)$ the map 
\[
 G \to s^m \pa_s^m (\Rad G)(\theta,s)
\]
from
$\Ccd$ to $\S_e(S^{n-1} \times \R)$ has an extension as a continuous linear bijection from $L^2(\R^n)$ to
$L^2_e(S^{n-1} \times \R, s^{-2m})$. So from proposition \ref{prop:reflection}, the map 
\[
g(x) \to G(X) \to s^m \pa_s^m (\Rad G)(\theta,s)
\]
from $\Ccd$ to $\S_e(S^{n-1} \times \R)$ has an
extension as a continuous linear bijection from $L^2(\R^n, |x|^2)$ to $L_e^2(S^{n-1} \times \R, s^{-2m})$.

Given $h : \R^n \to \R$, if we define $\tilde{h} : S^{n-1} \times (\R \setminus \{0\}) \to \R$ with
$\tilde{h}(\theta,s) = h(\theta/(2s))$ then, as shown in the proof of Theorem \ref{thm:isometry} for $\U$, the
map $ h \to \tilde{h}$ is a continuous linear bijection from $L^2(\R^n, |x|^{-2})$ to 
$L^2_e(S^{n-1} \times \R, s^{-2m})$. 

From (\ref{eq:RGVg}) we have 
\[
s^m \pa_s^m (\Rad G)(\theta,s) 
= c  \, r^{2-m} \, \pa_r ( r^m (\V g)(r \theta) )|_{r = 1/(2s)} , \qquad \forall s>0, ~ |\theta|=1,
\]
so using the results in the previous two paragraphs, the isometry $g \to  r^{2-m} \, \pa_r ( r^m (\V g)(r \theta) )$ has a 
continuous
linear extension which is a bijection from $L^2(\R^n, |x|^2)$ to $L^2(\R^n, |x|^{-2})$. This is equivalent to the 
statement associated
with $\V$ in Theorem \ref{thm:isometry}.

%%%%%%%%%%%%%%%%

\subsection{Proof of theorem \ref{thm:first} for $\V$}
%%%%%%%%%%%%%%%%%%%%

 If $g \in \Ccd$ and we define $G(X) = g(X/|X|^2) \, |X|^{-n-1}$ then $G \in \Ccd$ and from (\ref{eq:RGVgs}) we have
 \[
 \pa_s^{m-1}(\Rad G)(\theta,s) =\frac{(-1)^{m-1} 2\pi ^m}{s^m} (\V g)(\frac{\theta}{2s}),
\qquad \forall \theta\in S^{n-1}, s>0.
\]
For $s<0$, noting that $(\Rad G)(\theta,s) = (\Rad G)(-\theta,-s)$ we have
\begin{align*}
\pa_s^{m-1}(\Rad G)(\theta,s) 
& = (-1)^{m-1} (\pa_s^{m-1}(\Rad G))(-\theta,-s)
\\
& = \frac{2 \pi^m}{ (-s)^m} (\V g)(-\theta/(-2s) )
\\
& =  \frac{2 \pi^m}{ (-s)^m} (\V g)(\theta/(2s) ),
\end{align*}
so
\[
\pa_s ^{m-1}(\Rad G)(\theta,s) =
\frac{(-1)^{m-1}2\pi^{m}}{s^{m-1}|s|}(\V g)(\theta/(2s) ),
\qquad \forall ~s \in \R, ~s \neq 0, ~|\theta|=1
\]
hence
\[
\partial_s ^{n-1}(\Rad G)(\theta,s) 
= \partial_s ^{m+1}\left( \frac{(-1)^{m-1}2\pi^{m}}{s^{m-1}|s|}(\V g)(\theta/(2s)) \right),
 \qquad \forall ~s \in \R, ~s \neq 0, ~|\theta|=1.
\]
Using the Radon transform inversion formula (see Theorem \ref{thm:radon}), for any $X \neq 0$, we have
\begin{align*}
G(X) 
& =\frac{(-1)^m}{2(2\pi)^{2m}}\int_{S^{n-1}} \pa_s ^{n-1}(\Rad G)(\theta,s)|_{s=X\cdot \theta}\; d\theta
\\
&= -\frac{1}{(4\pi)^m}\int_{S^{n-1}} \pa_s ^{m+1}
\left( \frac{ (\V g)(\theta/(2s) ) }{s^{m-1}|s|} \right)\bigg|_{s=X\cdot\theta}\; d\theta,
\end{align*}
implying
\begin{align*}
g(x) 
& = \frac{-1}{ (4\pi)^m |x|^{n+1} }\int_{S^{n-1}} \pa_s ^{m+1}
\left( \frac{ (\V g)(\theta/(2s) ) }{s^{m-1}|s|} \right)\bigg|_{s=x\cdot\theta/|x|^2}\; d\theta, 
\qquad \forall x \in \R^n, ~ x \neq 0,
\end{align*}
which proves the theorem.

%%%%%%%%%%%%%%%%%%%%%%%%%%%%%%%%%%%%%
\subsection{Proof of theorem \ref{thm:second} for $\V$}

We plan to use the isometry of $\V$ to give another inverse for $\V$. Towards that, we construct the adjoint of $\V$ 
associated with the inner products suggested by the isometry of $\V$. 
\begin{prop}[\bf The adjoint of $\mathcal{V}$]\label{prop:adjointvn} For odd $n=2m+1$ we have
\beqn
\int_{\R^n} |x|^2 \,  \phi(x) \, |x|^{-m} \pa_r ( |x|^m \, (\V g)(x) ) \,  dx 
= \int_{\R^n} |x|^2  \, (\V^* \phi )(x) \, g(x) \, dx,
\qquad \forall \, g, \phi  \in \Ccd
\eeqn
where
\[
(\V^* \phi)(x) = \frac{1}{4  (-2 \pi)^m |x|^{m+1} } ( \pa_s^m \Rad \phi^*)(x/|x|, |x|/2) 
 \, \phi(x) \, dx, \qquad x \neq 0
\]
where $\phi^*(x)=|x| \, \phi(x)$.
Further
\[
\int_{\R^n} |x|^2  \,  | \phi(x)|^2 \, dx = 8 \int_{\R^n} |x|^2   |(V^* \phi)(x)|^2 \, dx, \qquad 
\forall \phi \in \Ccd
\]
and the map $\V^* : \Ccd \to L^2(\R^n, |x|^2)$ has a continuous linear extension as a {\bf bijection} from
$L^2(\R^n, |x|^2)$ to itself.

\end{prop}

%%%%%%%%%%%%%%
%
\begin{proof}
 Let $\phi, g \in \Ccd$. Then, from (\ref{eq:RGVg}) in  the proof of Theorem \ref{thm:isometry} for $\V$, we have
 \[
4(-2 \pi) ^m  |x|^2 \, \pa_r ( r^m (\V g)(x) ) = \pa_s^m (\Rad G)(\theta, 1/(2r)), \qquad \forall x \neq 0
\]
where $r=|x|$, $\theta=x/|x|$ and
\[
G(Y) = |Y|^{-n-1} g(Y/|Y|^2).
\]
Hence, using the substitution $Y = y/|y|^2$ and the homogeneity of the delta function, we have
\begin{align*}
4 (-2 \pi)^m \int_{\R^n} & |x|^2 \,  \phi(x) \, |x|^{-m} \pa_r ( |x|^m \, (\V g)(x) ) \,  dx 
\\
&=
\int_{\R^n} |x|^{-m} \, \phi(x) \, \pa_s^m (\Rad G)(\theta, 1/(2r)) \, dx
\\
& = \int_{\R^n} |x|^{-m} \, \phi(x) \, \int_{\R^n} \delta^m( 1/(2r) - Y \cdot \theta) \, G(Y) \, dY \, dx
\\
& = \int_{\R^n} |x|^{-m} \, \phi(x) \, \int_{\R^n} \delta^m( 1/(2r) - y \cdot \theta/|y|^2 ) \, |y|^{n+1} g(y) \, |y|^{-2n} \, dy \, dx
\\
& = 2^{m+1} \int_{\R^n} |x|^{-m} \, r^{m+1} \, \phi(x) \, \int_{\R^n} \delta^m( |y|^2 - 2x \cdot y) \,  |y|^{2m+2} \, |y|^{1-n} 
 g(y) \, dy \, dx
 \\
 & = 2^{m+1} \int_{\R^n} |x| \,  \phi(x) \, \int_{\R^n} \delta^m( |y|^2 - 2x \cdot y) \,  g(y) \, |y|^2 \, dy \, dx
 \\
 & = 4 (-2 \pi)^m \int_{\R^n} |y|^2 \, (\V^* \phi)(y) \, g(y) \, dy
\end{align*}
where, for $y \neq 0$, we have
\begin{align*}
(\V^* \phi)(y)  
& = \frac{1}{2 (-\pi)^m} \int_{\R^n} |x| \, \delta^m( |y|^2 - 2x \cdot y) \, \phi(x) \, dx
\\
& = \frac{1}{4  (-2 \pi)^m |y|^{m+1} } \int_{\R^n}  \delta^m( |y|/2 - x \cdot y/|y|) \, |x| \, \phi(x) \, dx
\\
& = \frac{1}{4  (-2 \pi)^m |y|^{m+1} } ( \pa_s^m \Rad \phi^*)(y/|y|, |y|/2) 
\end{align*}
where $\phi^*(x) = |x| \phi(x)$.

Now, using the isometry of the Radon transform stated in Theorem \ref{thm:radon}, we have
\begin{align*}
4^2 (2 \pi)^{2m} & \int_{\R^n}  |x|^2 \, |(\V^* \phi)(x)|^2 \, dx
 = \int_{\R^n} |x|^{-2m} \, | (\pa_s^m \Rad \phi^*)(x/|x|, |x|/2) |^2 \, dx
\\
& = \int_{|\theta|=1} \int_0^\infty |(\pa_s^m \Rad \phi^*)(\theta, r/2)|^2 \, dr \, d \theta
 = 2 \int_{|\theta|=1} \int_0^\infty  |(\pa_s^m \Rad \phi^*)(\theta, s)|^2 \, ds \, d \theta
\\
& = \int_{|\theta|=1} \int_\R |(\pa_s^m \Rad \phi^*)(\theta, s)|^2 \, ds \, d \theta
 = 2 (2 \pi)^{2m} \int_{\R^n} |\phi^*(x)|^2 \, dx
\\
& = 2 (2 \pi)^{2m} \int_{\R^n} |x|^2 \, |\phi(x)|^2 \, dx 
\end{align*}
which proves the isometry for $\V^*$.

The result about the range of the extension of $\V^*$ also follows easily from the range of $\Rad$.

\end{proof}

%%%%%%%%%%%%%%%%%%%

The proof of Theorem \ref{thm:second} for the $\V$ case is just an imitation of the proof for the $\U$ case.

%% file: mean.tex
% !TEX root = driver.tex
%

\section{Proof of theorem \ref{thm:mean}}\label{sec:mean}

First, as in proofs of the other theorems, we express the spherical averages over spheres through the origin as a Radon transform. Given $h \in \Ccd$ we define $H(X)$ as
\[
H(X)  =|X|^{2-2n} \, h(X/|X|^2), \qquad X \in \R^n, ~ X \neq 0;
\]
note $H \in \Ccd$ from proposition \ref{prop:reflection}. For $x \neq 0$,
using the substitution $Y=y/|y|^2$ and the homogeneity of $\delta(s)$, we have
\begin{align}
(\M h)(x,|x|) &= \frac{1}{\omega_{n-1} |x|^{n-1}} \int_{|y-x|=|x|} h(y)\; dS_y
\nn
\\
& = \frac{2}{\omega_{n-1}|x|^{n-2}}\int_{\R^n} h(y) \delta(|y-x|^2-|x|^2 )\; dy 
\nn
\\
&=\frac{2}{\omega_{n-1}|x|^{n-2}}\int_{\R^n} h(y) \delta(|y|^2-2y \cdot x)\; dy 
\nn
\\
&=\frac{2}{\omega_{n-1} |x|^{n-2}} \int_{\R^n} h(Y/|Y|^2) \, \delta(|Y|^{-2} - 2 |Y|^{-2} Y \cdot x)\; |Y|^{-2n} dY
\nn
\\
&=\frac{1}{\omega_{n-1} |x| ^{n-1}}\int_{\R^n}H(Y) \, \delta(Y \cdot x/|x| - 1/(2|x|) )\; dY
\nn
\\
&=\frac{1}{\omega_{n-1} |x|^{n-1}} (\Rad H)(x/|x|,1/(2|x|)).
\nn
\end{align}
We may rewrite this as
\beqn
(\Rad H)(\theta, 1/(2\rho)) = \omega_{n-1} \rho^{n-1} (\M h)( x, |x|), \qquad x \in \R^n, ~x \neq 0
\label{eq:xrt}
\eeqn
where $\rho=|x|$ and $\theta=x/|x|$, or as
\beqn
(\Rad H)(\theta, s) =  \omega_{n-1} 2^{1-n} s^{1-n} (\M h)( \theta/(2s), 1/(2s) ), \qquad s>0, ~ |\theta|=1.
\label{eq:temp21}
\eeqn
Note that these relations are true for odd and even $n$.

For odd $n$, using Theorem \ref{thm:radon} and the substitution $x=X/|X|^2$ we have
\begin{align*}
\int_{\R^n} |x|^{2n-4} \, |h(x)|^2 \, dx 
& = \int_{\R^n} |X|^{4-2n} \, | h(X/|X|^2) |^2 \, |X|^{-2n} \, dX
\\
& = \int_{\R^n} |H(X)|^2 \, dX
\\
& = \frac{1}{2(2\pi)^{n-1}}\int_{S^{n-1}} \int_{-\infty}^\infty | \partial_s ^{\frac{n-1}{2}}(\Rad H)(\theta,s) |^2\; ds\; d\theta.
\\
&=\frac{1}{(2\pi)^{n-1}} \int_{S^{n-1}} \int_{0}^\infty | \partial_s ^{\frac{n-1}{2}}(\Rad H)(\theta,s) |^2\; ds\; d\theta.
\end{align*}
So using $\omega_{n-1} = 2 \pi^{n/2}/\Gamma(n/2)$, the substitution $s=1/(2 \rho)$, noting 
$\pa_s = -2 \rho^2 \pa_\rho$, and (\ref{eq:xrt}), we obtain
\begin{align*}
\int_{\R^n} |x|^{2n-4} \, |h(x)|^2 \, dx 
& = \frac{ 2^{n-2}}{ (2 \pi)^{n-1} } \int_{S^{n-1}} \int_0^\infty 
| (\rho^2 \pa_\rho)^{(n-1)/2} (\Rad H)(\theta, 1/(2 \rho) ) |^2 \, \rho^{-2} \, d \rho \, d \theta
\\
& = \frac{ \omega_{n-1}^2 }{ 2 \pi^{n-1} }
 \int_{S^{n-1}} \int_0^\infty | ( \rho^2 \pa_\rho)^{(n-1)/2} ( \rho^{n-1} (\M h)(y, |y|) ) |^2  \, \rho^{-2} \, d \rho \, d \theta
 \\
 & = \frac{2 \pi}{\Gamma(n/2)^2}  \int_{S^{n-1}} \int_0^\infty | (\rho^2 \pa_\rho)^{(n-1)/2} ( \rho^{n-1} (\M h)(y, |y|) ) |^2  \,
  \rho^{-2} \, d\rho \, d \theta
  \\
  & = \frac{2 \pi}{\Gamma(n/2)^2}  \int_{\R^n}  | (\rho^2 \pa_\rho)^{(n-1)/2} ( \rho^{n-1} (\M h)(y, |y|) ) |^2  \, |y|^{-n-1} \, dy
\end{align*}
where $y=\rho \theta$ on the RHS. This proves the isometry.

Next we prove the inversion formula.
Using  $(\Rad H)(-\theta,-s) = (\Rad H)(\theta, s)$ in the Radon inversion formula for odd $n$ in Theorem \ref{thm:radon}, we 
have
\begin{align*}
\frac{2(2\pi)^{n-1}}{(-1)^{\frac{n-1}{2}}} H(X) & = \int_{S^{n-1}} \pa_s^{n-1} (\Rad H)(\theta, s)|_{s=X\cdot \theta}\; d\theta.
\\
& =   \int_{X \cdot \theta >0} \pa_s^{n-1} (\Rad H)(\theta, s)|_{s=X\cdot \theta}\; d\theta
+   \int_{X \cdot \theta <0} \pa_s^{n-1} (\Rad H)(\theta, s)|_{s=X\cdot \theta}\; d\theta
\\
& =   \int_{X \cdot \theta >0} \pa_s^{n-1} (\Rad H)(\theta, s)|_{s=X\cdot \theta}\; d\theta
+   \int_{X \cdot \theta <0} \pa_s^{n-1} (\Rad H)(-\theta, -s)|_{s=X\cdot \theta}\; d\theta
\\
& =   \int_{X \cdot \theta >0} \pa_s^{n-1} (\Rad H)(\theta, s)|_{s=X\cdot \theta}\; d\theta
+   \int_{X \cdot \theta >0} \pa_s^{n-1} (\Rad H)(\theta, -s)|_{s=-X\cdot \theta}\; d\theta
\\
& = 2  \int_{X \cdot \theta >0} \pa_s^{n-1} (\Rad H)(\theta, s)|_{s=X\cdot \theta}\; d\theta.
\end{align*}
Using this,  the substitution $s=1/(2 \rho)$, noting 
$\pa_s = -2 \rho^2 \pa_\rho$,  and (\ref{eq:temp21}) we have
\begin{align}
\frac{(2\pi)^{n-1}}{(-1)^{\frac{n-1}{2}}} h(x) & = \frac{(2\pi)^{n-1}}{(-1)^{\frac{n-1}{2}}} \, |x|^{2-2n} H(x/|x|^2) 
\nn
\\
& = |x|^{2-2n}  \int_{x \cdot \theta >0} \pa_s^{n-1} (\Rad H)(\theta, s)|_{s=x\cdot \theta/|x|^2 }\; d\theta
\nn
\\
& = 2^{1-n} \omega_{n-1} \, |x|^{2-2n}  \int_{x \cdot \theta >0} \partial_ s^{n-1} ( s^{1-n}  (\M h)(\theta/(2s), 1/(2s) ) ) 
|_{s=x\cdot \theta/|x|^2}\; d\theta
\nn
\\
& = 2^{n-1} \omega_{n-1} \, |x|^{2-2n} \,
\int_{x \cdot \theta >0} ( \rho^2 \partial_ \rho)^{n-1} ( \rho^{n-1} 
 (\M h)( \rho \theta, \rho ) ) |_{\rho=|x|^2/(2x \cdot \theta)}\; 
d\theta.
\label{eq:temp77}
\end{align}

Now, using $y= \rho \theta$, we have
\begin{align*}
\int_{x \cdot \theta >0} f(\rho, \theta)|_{\rho=|x|^2/(2 x \cdot \theta)} \, d \theta
& =
\int_{x \cdot \theta >0} \int_0^\infty  f(\rho, \theta) \, \delta( \rho - |x|^2/(2 x \cdot \theta) ) \, d \rho \, d \theta
\\
& = \int_{|\theta|=1} \int_0^\infty 2 (x \cdot \theta) \,  f(\rho, \theta) \, \delta( 2 x \cdot (\rho \theta) - |x|^2 ) \, d \rho \, d \theta
\\
& = \int_{|\theta|=1} \int_0^\infty \rho ^{ -n} \, 2 (x \cdot \rho \theta) \,  f(\rho, \theta) \, 
\delta( 2 x \cdot (\rho \theta) - |x|^2 ) \, \rho^{n-1}
\, d \rho \, d \theta
\\
& = \int_{\R^n} 2 (x \cdot y) \, |y|^{-n} \, f(\rho, \theta) \, \delta( 2 x \cdot y - |x|^2) \, dy
\\
& = |x|^2 \,  \int_{\R^n} |y|^{-n} \, f(\rho, \theta) \, \delta( 2 x \cdot y - |x|^2) \, dy
\\
& = \frac{|x|}{2} \int_{2 y \cdot x = |x|^2} |y|^{-n} \, f(\rho, \theta) \, dS_y.
\end{align*}
Hence, using (\ref{eq:temp77}) (below $\rho = |y|$)
\begin{align*}
h(x) = \frac{ (-1)^{(n-1)/2} \omega_{n-1}}{ 2\, \pi^{n-1}} |x|^{3-2n} \,
\int_{2 y \cdot x = |x|^2}  |y|^{-n} \, ( \rho^2 \partial_ \rho)^{n-1} ( \rho^{n-1}  (\M h)( y, |y|) ) \, dS_y.
\end{align*}

%% file: mybib.tex
% !TEX root = driver.tex
%
% This is the Bibliography file
%
% For 100-999 change 99 to 999; for 1000-9999 change 99 to 9999 